\documentclass{amsart}

\author     [K. Shimizu]{Kenichi Shimizu}
\address    [K. Shimizu]{Institute of Mathematics, University of Tsukuba. Tsukuba, Ibaraki, 305-8571, Japan.}
\title      [Frobenius-Schur indicators in Tambara-Yamagami categories]
            {Frobenius-Schur indicators \\ in Tambara-Yamagami categories}
\date       {May 24, 2010}

\usepackage{amsmath}
\usepackage{amssymb}
\usepackage{amscd}
\usepackage{amsthm}

\newtheorem{counter}            {}[section]
\theoremstyle{definition}
\newtheorem{definition}         [counter]{Definition}
\theoremstyle{plain}
\newtheorem{lemma}              [counter]{Lemma}
\newtheorem{proposition}        [counter]{Proposition}
\newtheorem{theorem}            [counter]{Theorem}
\newtheorem{corollary}          [counter]{Corollary}

\theoremstyle{remark}
\newtheorem{remark}             [counter]{Remark}

\newcommand{\sgn}       {\mathop{\rm sgn}\nolimits}

\newcommand{\id}        {\mathop{\rm id}\nolimits}

\newcommand{\Hom}       {\mathop{\rm Hom}\nolimits}
\newcommand{\End}       {\mathop{\rm End}\nolimits}
\newcommand{\Ker}       {\mathop{\rm Ker}\nolimits}
\newcommand{\Rad}       {\mathop{\rm Rad}\nolimits}
\newcommand{\Aut}       {\mathop{\rm Aut}\nolimits}
\newcommand{\Irr}       {\mathop{\rm Irr}\nolimits}
\newcommand{\Rep}       {\mathop{\rm Rep}\nolimits}

\newcommand{\Trace}     {\mathop{\rm Tr}\nolimits}

\newcommand{\ptrace}    {\mathop{\rm ptr}\nolimits}
\newcommand{\pdim}      {\mathop{\rm pdim}\nolimits}
\newcommand{\FPdim}     {\mathop{\rm FPdim}\nolimits}

\newcommand{\legendre}  [2]{{\genfrac{(}{)}{1pt}{}{#1}{#2}}}
\newcommand{\TY}        {\mathcal{TY}}
\newcommand{\eval}      {b}
\newcommand{\coev}      {d}

\newenvironment{enumalph}
{\begin{enumerate}\renewcommand{\labelenumi}{\textnormal{\hbox to 1.4em{(\hfil\alph{enumi}\hfil)}}}}
  {\end{enumerate}}

\begin{document}

\begin{abstract}
  We introduce formulae of Frobenius-Schur indicators of simple objects of Tambara-Yamagami categories. By using techniques of the Fourier transform on finite abelian groups, we study some arithmetic properties of indicators.
\end{abstract}

\maketitle

\section{Introduction}
\label{sec:introduction}

At first Frobenius-Schur indicators were introduced in the group theory to determine whether a given irreducible representation admits a symmetric or a skew-symmetric invariant bilinear form. After that the theory of indicators were generalized to semisimple Hopf algebras by Linchenko and Montgomery \cite{MR1808131} and to semisimple quasi-Hopf algebras by Mason and Ng \cite{MR2104908}.

Since representations of a (quasi-)Hopf algebra form a monoidal category, it is natural to consider any categorification of the theory of Frobenius-Schur indicators. Ng and Schauenburg developed the theory of Frobenius-Schur indicators in linear pivotal categories, which are linear monoidal categories equipped with certain additional structures, see \cite{MR2313527} and \cite{MR2381536}. They defined the $n$-th Frobenius-Schur indicator $\nu_n(V)$ of an object $V$ of such categories to be the trace of certain linear automorphism on $\Hom(1, V^{\otimes n})$, where $1$ is the neutral object.

Now, it is interesting to obtain formulae for indicators of objects of categories that are not coming from Hopf algebras. In this paper, we introduce closed formulae for indicators of objects of Tambara-Yamagami categories \cite{MR1659954}, which are a well-studied class of semisimple linear pivotal categories. We also study properties of indicators and give some applications of our formulae.

This paper is organized as follows. In Section~\ref{sec:preliminaries}, we recall definitions and results from the theory of monoidal categories and Frobenius-Schur indicators. We also introduce techniques of the Fourier transform, which will be important tools for analyzing arithmetic properties of indicators.

In Section~\ref{sec:indicators}, we recall the definition of Tambara-Yamagami categories (Definition~\ref{def:TY}) and determine all Frobenius-Schur indicators of all simple objects. Our formulae of indicators (Theorem~\ref{thm:fs-ind-1} and~\ref{thm:fs-ind-ty-2}) seem to be of complicated forms, however, by using techniques of the Fourier transform, it turns out that values of indicators are very familiar form, see Theorem~\ref{thm:fs-ind-arith}.

For a pivotal fusion category $\mathcal{C}$, the author \cite{S10} has introduced the symbol $\nu_n(\mathcal{C})$ as a certain weighted sum of $n$-th Frobenius-Schur indicators of simple objects of $\mathcal{C}$. In Section~\ref{sec:indicator-sum}, we study $\nu_n(\mathcal{C})$ for Tambara-Yamagami categories $\mathcal{C}$. We also argue and give some partial results on Frobenius theorem for $\mathcal{C}$ (see Definition~\ref{def:frobenius}), which is motivated by the classical theorem of Frobenius.

Kashina, Montgomery and Ng in \cite{KMN09} mentioned whether there exists a semisimple Hopf algebra such that the trace of its antipode is zero. In Section~\ref{sec:traceless}, as an application of our result, we give such examples. Note that, if $\mathcal{C}$ is the category of representations of a semisimple Hopf algebra $H$, the above $\nu_2(\mathcal{C})$ is equal to the trace of the antipode. By using our results, we can easily get fusion categories $\mathcal{C}$ with $\nu_2(\mathcal{C}) = 0$. Combining this observation with Tambara's result \cite{MR1776075} on fiber functors on Tambara-Yamagami categories, we obtain semisimple Hopf algebras with the traceless antipode.

In Section~\ref{sec:computation}, as computational examples, we determine Frobenius-Schur indicators of objects of Tambara-Yamagami categories associated with finite vector spaces. They are expressed explicitly by using several Legendre symbols. We see that all such categories can be distinguished by using Frobenius-Schur indicators.

\section{Preliminaries}
\label{sec:preliminaries}

Throughout this paper, we work over the field $\mathbb{C}$ of complex numbers. Every linear category is assumed to be with finite-dimensional $\Hom$-spaces (over $\mathbb{C}$). Unless otherwise noted, functors between $\mathbb{C}$-linear categories are always assumed to be $\mathbb{C}$-linear. We denote by $\mu_n$ the set of $n$-th roots of unity in $\mathbb{C}$. The set of all roots of unity is denoted by $\mu_\infty$ in $\mathbb{C}$.

All (quasi-)Hopf algebras are assumed to be finite-dimensional over $\mathbb{C}$. We denote by $\Rep(H)$ the $\mathbb{C}$-linhear monoidal category of finite-dimensional representations of a quasi-Hopf algebra $H$.

\subsection{Fusion categories}

We will freely use the basic theory of monoidal categories. The reader may refer to \cite{MR1797619}, \cite{MR2183279}, \cite{MR1321145} and \cite{MR1966525} for related topics. However, for reader's convenience, we recall some definitions and facts in this subsection.

First we fix some conventions. Let $\mathcal{C}$ be a monoidal category. The associativity constraint is denoted by $\Phi_{X,Y,Z}: (X \otimes Y) \otimes Z \to X \otimes (Y \otimes Z)$. The left dual object of $V \in \mathcal{C}$ is denoted by $V^*$ if it exists. The evaluation and the coevaluation are usually denoted by $\coev_V: 1 \to V \otimes V^*$ and $\eval_V: V^* \otimes V \to 1$, where $1$ denotes the unit object of $\mathcal{C}$.

Suppose that $\mathcal{C}$ is left rigid, that is, every object of $\mathcal{C}$ has a left dual. Then the assignment $V \mapsto V^{**}$ gives rise to a monoidal functor $(-)^{**}: \mathcal{C} \to \mathcal{C}$. A {\em pivotal structure} on $\mathcal{C}$ is an automorphism $j: \id_{\mathcal{C}} \to (-)^{**}$ of monoidal functors. A {\em pivotal category} is a rigid monoidal category equipped with a pivotal structure. Given such a $j$, the (right) {\em pivotal trace} of $f: V \to V$ in $\mathcal{C}$ is defined and denoted by
\begin{equation*}
  \ptrace_j(f) = \eval_{V^*}^{} (j_V f \otimes \id_V) \coev_V^{}: 1 \to 1.
\end{equation*}
We call $\pdim_j(V) = \ptrace_j(\id_V)$ the (right) {\em pivotal dimension} of $V \in \mathcal{C}$. We will often omit the subscript $j$ when it is obvious. A pivotal structure $j$ is {\em spherical} \cite{MR1686423} if $\ptrace_j(f) = \ptrace_j(f^*)$ for every $f$. A {\em spherical category} is a rigid monoidal category equipped with a spherical pivotal structure.

Let $\Irr(\mathcal{C})$ denotes the set of representatives of isomorphism classes of an (skeletally small) abelian category $\mathcal{C}$. A {\em fusion category} is a $\mathbb{C}$-linear semisimple abelian rigid monoidal category $\mathcal{C}$ with finite $\Irr(\mathcal{C})$ such that the unit object $1 \in \mathcal{C}$ is simple and $\End(V) \cong \mathbb{C}$ for every $V \in \Irr(\mathcal{C})$. Until the end of this subsection, we assume $\mathcal{C}$ to be a fusion category. Then, a pivotal structure $j$ on $\mathcal{C}$ is spherical if and only if $\pdim_j(V) \in \mathbb{R}$ for every $V \in \Irr(\mathcal{C})$.

The {\em Frobenius-Perron dimension} of $V \in \mathcal{C}$, denoted by $\FPdim(V)$, is the largest real eigenvalue of the left multiplication of $V$ on the Grothendieck ring $K(\mathcal{C})$. A {\em canonical pivotal structure} is a pivotal structure $j$ on $\mathcal{C}$ such that $\pdim_j(V) = \FPdim(V)$ for every $V \in \mathcal{C}$. Such a structure exists if and only if $\mathcal{C}$ is pseudo-unitary in the sense of \cite{MR2183279}.

We denote by $\mathcal{Z}(\mathcal{C})$ the left Drinfeld center of $\mathcal{C}$ (which has as objects the pairs $(V, e_V)$ of an object $V \in \mathcal{C}$ and a left half-braiding $e_V: V \otimes (-) \to (-) \otimes V$). Also $\mathcal{Z}(\mathcal{C})$ is a fusion category, under our assumption that $\mathcal{C}$ is a fusion category. The assignment $(V, e_V) \mapsto V$ extends to a monoidal functor $\mathcal{Z}(\mathcal{C}) \to \mathcal{C}$. This functor has a two-sided adjoint $I: \mathcal{C} \to \mathcal{Z}(\mathcal{C})$ such that
\begin{equation}
  \label{eq:adjoint}
  I(V) \cong \bigoplus_{(X, e_X) \in \Irr(\mathcal{Z}(\mathcal{C}))} (X, e_X)^{\oplus \dim \Hom(V,X)}
  \quad (V \in \mathcal{C}).
\end{equation}

Given a pivotal structure $j$ on $\mathcal{C}$, one can define a pivotal structure $J$ on $\mathcal{Z}(\mathcal{C})$ so that $\pdim_J((V, e_V)) = \pdim_j(V)$. The induced structure $J$ is spherical if and only if so is $j$. Note that there exists a bijection between twists (in the sense of \cite[Definition~XIV.3.2]{MR1321145}) of a braided fusion category and spherical pivotal structures on it. Therefore $\mathcal{Z}(\mathcal{C})$ is naturally a ribbon category if $\mathcal{C}$ is a spherical fusion category.

\subsection{Frobenius-Schur indicators}

Let $\mathcal{C}$ be a rigid monoidal category. For an object $V \in \mathcal{C}$, define $V^{\otimes n} \in \mathcal{C}$ inductively by $V^{\otimes 0} = 1$, $V^{\otimes 1} = V$ and $V^{\otimes n} = V \otimes V^{\otimes (n - 1)}$ ($n \ge 2$). It is well-known that there exist isomorphisms
\begin{align*}
  A^X_{Y, Z}   : \Hom(X, Y \otimes Z) & \to \Hom(Y^* \otimes X, Z) \quad \text{and} \\
  B^{X, Y}_{Z} : \Hom(X \otimes Y, Z) & \to \Hom(X, Z \otimes Y^*)
\end{align*}
that are natural in $X, Y, Z \in \mathcal{C}$ \cite[XIV.2.2]{MR1321145}. Now we suppose that $\mathcal{C}$ has a pivotal structure $j$. Then linear automorphisms $E_V^{(n)}$ on $\Hom(1, V^{\otimes n})$ is defined by
\begin{equation}
  \label{eq:E-map}
  E_V^{(n)}(f) = \Phi \circ (\id_{V^{\otimes (n - 1)}} \otimes j_V^{-1})
  \circ B^{1, V^*}_{V^{\otimes(n-1)}} \circ A^{1}_{V,V^{\otimes(n-1)}}(f)
\end{equation}
where $\Phi$ is the associativity isomorphism $V^{\otimes (n - 1)} \otimes V \to V^{\otimes n}$. The {\em $n$-th Frobenius-Schur indicator} $\nu_n(V)$ of $V \in \mathcal{C}$ is given and denoted by
\begin{equation*}
  \nu_n(V) = \Trace \left( E_V^{(n)} \right),
\end{equation*}
where $\Trace$ means the usual trace of linear maps.

Suppose that $\mathcal{C}$ is a spherical fusion category. Then, as remarked above, also $\mathcal{Z}(\mathcal{C})$ is spherical and it has a canonical twist $\theta$. Ng and Schauenburg showed that
\begin{equation*}
  \nu_n(V) = \frac{1}{\dim(\mathcal{C})} \ptrace \left( \theta_{I(V)}^n \right) \quad (V \in \mathcal{C}),
\end{equation*}
where $\dim(\mathcal{C})$ is the global dimension in the sense of \cite[Definition~2.2]{MR2183279} and $I$ is a two-sided adjoint of the forgetful functor $\mathcal{Z}(\mathcal{C}) \to \mathcal{C}$. In view of~(\ref{eq:adjoint}), we have
\begin{equation}
  \label{eq:FS-ind-twist}
  \nu_n(V) = \frac{1}{\dim(\mathcal{C})}
  \sum_{(X, e_X) \in \Irr(\mathcal{Z}(\mathcal{C}))} \theta_X^n \pdim(X) \dim_\mathbb{C}(\Hom(V, X)).
\end{equation}

\subsection{Bicharacters of finite abelian groups}

We recall some definitions from the group theory. A {\em bicharacter} of an abelian group $A$ is a map $\chi: A \times A \to \mathbb{C}^\times$ that is multiplicative in each variable. It is said to be {\em symmetric} if $\chi(a, b) = \chi(b, a)$ for all $a, b \in A$, and is said to be {\em alternating} if $\chi(a, a) = 1$ for all $a \in A$.

Let $\chi$ be a bicharacter of a finite group $A$. It is clear that $\chi(a, b) \in \mu_n$, where $n$ is the greatest common divisor of the order of $a$ and that of $b$. In particular, $\chi$ takes values in $\mu_N$, where $N$ is the exponent of $A$, that is, the least common multiple of orders of all elements of $A$

Suppose that $A$ is finite. Then the group of $\mathbb{C}$-valued characters of $A$ is denoted by $A^\vee$. Given a bicharacter $\chi$ of $A$, a group homomorphism $\chi^\natural: A \to A^\vee$ is defined by $\chi^\natural(a)(x) = \chi(x, a)$. The {\em radical} of $\chi$ is defined and denoted by
\begin{equation*}
  \Rad(\chi) := \Ker(\chi^\natural) = \{ a \in A \mid \text{$\chi(x, a) = 1$ for all $x \in A$} \}.
\end{equation*}
By the orthogonality relation of characters, we have
\begin{equation}
  \label{eq:orth-rel}
  \frac{1}{|A|} \sum_{a \in A} \chi(a, z) =
  \begin{cases}
    1 & \text{if $z \in \Rad(\chi)$}, \\
    0 & \text{otherwise}.
  \end{cases}
\end{equation}
We say that $\chi$ is {\em non-degenerate} if $\Rad(\chi)$ is trivial.

Recall that a bicharacter of $A$ is a coboundary if and only if it is symmetric (see, e.g., \cite{MR0174656}). This means that, if $\chi$ is a symmetric bicharacter of $A$, there exists a function $\rho: A \to \mathbb{C}^\times$ satisfying
\begin{equation}
  \label{eq:coboundary}
  \chi(a, b) = \partial \rho(a, b) := \rho(a) \rho(a b)^{-1} \rho(b) \quad (a, b \in A).
\end{equation}
Throughout this paper, the set of such functions is denoted by $C(\chi)$. We note that $A^\vee$ acts freely and transitively on $C(\chi)$ by the multiplication of functions on $A$. In particular, $|C(\chi)| = |A^\vee| = |A|$.

Letting $a = b = 1$ in~(\ref{eq:coboundary}), we have $\rho(1) = 1$. By induction on $k$, we have
\begin{equation}
  \label{eq:coboundary-2}
  \rho(a_1) \rho(a_2) \cdots \rho(a_k) = \rho(a_1 a_2 \cdots a_k) \prod_{1 \le i < j \le k} \chi(a_i, a_j)
  \quad (a_1, \cdots, a_k \in A).
\end{equation}
This formula yields that $\rho$ takes values in $\mu_\infty$. In fact, if $a^n = 1$, then
\begin{equation*}
  \rho(a)^{2 n} = \rho(a^{2 n}) \cdot \chi(a, a)^{n (n - 1)} = \rho(1) \cdot \chi(1, a)^{n-1} = 1.
\end{equation*}
Note that $a^n = 1$ does not imply $\rho(a)^n = 1$ in general. We will encounter such an example in Section~\ref{sec:computation}. Provided that $|A|$ is odd, $a^n = 1$ implies $\rho(a)^n = 1$.

In this paper, we mean by a {\em pseudo-metric group}\footnote{Some authors call a pair $(E, q)$ of a finite abelian group $E$ and a non-degenerate quadratic form $q: E \to \mathbb{C}^\times$ a metric group. A metric group $(E, q)$ become a pseudo-metric group $(E, \chi)$ by letting $\chi = \partial q$. The word ``pseudo-metric'' does not have any topological implication.} a pair $(A, \chi)$ of a finite abelian group $A$ and a symmetric bicharacter $\chi$ of $A$. A morphism $f: (A, \chi) \to (A', \chi')$ of pseudo-metric groups is a group homomorphism $f: A \to A'$ satisfying $\chi \circ (f \times f) = \chi'$. Pseudo-metric groups form a category, say $\underline{\rm PMG}$. We call an isomorphism in this category {\em isometry}. We say that two symmetric bicharacters $\chi_1$ and $\chi_2$ of $A$ are {\em isometric} if $(A, \chi_1)$ and $(A, \chi_2)$ are isometric.

Given a finite number of pseudo-metric groups, we can define their product in an obvious way. Let $(A, \chi)$ be a pseudo-metric group and let $p_1, \cdots, p_m$ be all prime divisors of $|A|$. By the fundamental theorem of finite abelian groups and the Chinese remainder theorem, $(A, \chi)$ is decomposed into the product
\begin{equation*}
  (A, \chi) \cong (A_1, \chi_1) \times \cdots \times (A_m, \chi_m),
\end{equation*}
where $A_i$ is the Sylow $p_i$-subgroup of $A$ and $\chi_i$ is the restriction of $\chi$ to $A_i$.

Let $\chi$ be a symmetric bicharacter of $A$ and let $\sigma \in {\rm Gal}(\mathbb{C}/\mathbb{Q})$. Then it is obvious that also ${}^\sigma \chi := \sigma \circ \chi$ is a symmetric bicharacter of $A$. ${}^\sigma \chi$ is not isometric to $\chi$ in general, however, the following lemma holds:

\begin{lemma}
  \label{lem:bichar-gal}
  ${}^{\sigma\sigma}\chi$ is isometric to $\chi$.
\end{lemma}
\begin{proof}
  Let $N$ be the exponent of $A$. Fix a primitive $N$-th root $\zeta$ of unity. Then there exists a map $q: A \times A \to \mathbb{Z}_N$ such that $\chi(a, b) = \zeta^{q(a, b)}$. Recall that ${\rm Gal}(\mathbb{Q}[\zeta]/\mathbb{Q})$ is isomorphic to $\mathbb{Z}_N^\times$; This means that $\sigma(\zeta) = \zeta^s$ for some $s \in \mathbb{Z}_N^\times$, and hence
  \begin{equation*}
    {}^{\sigma\sigma} \chi(a, b) = \zeta^{s^2 q(a, b)} = \chi(a, b)^{s^2} = \chi(a^s, b^s)
    \quad (a, b \in A).
  \end{equation*}
  Since $s$ is relatively prime to $N$, the assignment $a \mapsto a^s$ gives an automorphism on $A$. Hence the above equation means that ${}^{\sigma\sigma}\chi$ is isometric to $\chi$.
\end{proof}

\begin{remark}
  Let $\sigma \in {\rm Gal}(\mathbb{C}/\mathbb{Q})$. Then the assignment $(A, \chi) \mapsto (A, {}^\sigma \chi)$ gives rise to an endofunctor $F_\sigma$ on $\underline{\rm PMG}$. In the same way as the proof of the above lemma, one can show that $F_\sigma \circ F_\sigma$ is isomorphic to the identity functor.
\end{remark}

\subsection{Fourier transform associated with bicharacters}

Let $L^2(X)$ denote the vector space of $\mathbb{C}$-valued functions on a finite set $X$. The {\em Fourier transform} on a finite abelian group $A$ is the linear map $\mathcal{F}: L^2(A) \to L^2(A^\vee)$ defined by
\begin{equation*}
  \mathcal{F}(f)(\lambda) = \frac{1}{\sqrt{|A|}} \sum_{x \in A} f(x) \lambda(x)^{-1}
  \quad (f \in L^2(A), \lambda \in A^\vee).
\end{equation*}
Fix a symmetric bicharacter $\chi$ of $A$. Set $\mathcal{F}_\chi := \chi^* \circ \mathcal{F}$, where $\chi^*: L^2(A^\vee) \to L^2(A)$ is the linear map induced from the group homomorphism $\chi^\natural: A \to A^\vee$. We call $\mathcal{F}_\chi$ the {\em Fourier transform associated with $\chi$}. More precisely,
\begin{equation*}
  \mathcal{F}_\chi(f)(a) = \frac{1}{\sqrt{|A|}} \sum_{x \in A} f(x) \chi(x, a)^{-1}
  \quad (f \in L^2(A), a \in A).
\end{equation*}

If $\chi$ is non-degenerate, then $\mathcal{F}_\chi$ is bijective with the inverse operator
\begin{equation*}
  \mathcal{F}^{-1}_\chi(f)(a) = \frac{1}{\sqrt{|A|}} \sum_{x \in A} f(x) \chi(x, a)
  \quad (f \in L^2(A), a \in A).
\end{equation*}
Note that we also deal with the case where $\chi$ is degenerate. In general, $\mathcal{F}_\chi$ is not bijective. The kernel and the image of $\mathcal{F}_\chi$ will be discussed later. It will turn out that $\mathcal{F}_\chi$ is bijective if and only if $\chi$ is non-degenerate.

The {\em convolution product} of $f, g \in L^2(A)$ is defined by
\begin{equation*}
  (f * g)(a) = \sum_{x \in A} f(x) g(x^{-1} a) \quad (a \in A).
\end{equation*}
$L^2(A)$ is a commutative associative algebra with respect to the convolution. We denote by $f^{*n}$ the $n$-fold iterated convolution product of $f \in L^2(A)$. The following identity is well-known and easy to prove:
\begin{equation}
  \label{eq:fourier-convolution}
  \mathcal{F}_\chi(f * g)(a) = \sqrt{|A|} \, \mathcal{F}_\chi(f)(a) \cdot \mathcal{F}_\chi(g)(a)
  \quad (f, g \in L^2(A), a \in A).
\end{equation}

Let $\delta_a \in L^2(A)$ be the delta function at $a \in A$, that is, the function on $A$ defined by $\delta_a(x) = \delta_{a, x}$ ($x \in A$). The set $\{ \delta_a \}_{a \in A}$ is a basis of $L^2(A)$. The matrix representation of $\mathcal{F}_\chi$ with respect to this basis is given as follows:
\begin{equation*}
  \mathcal{F}_\chi(\delta_a) = \frac{1}{\sqrt{|A|}} \sum_{b \in A} \chi(a, b)^{-1} \delta_b \quad (a \in A).
\end{equation*}
In particular,
\begin{equation}
  \label{eq:fourier-trace}
  \Trace(\mathcal{F}_\chi) = \frac{1}{\sqrt{|A|}} \sum_{a \in A} \chi(a, a)^{-1}.
\end{equation}

In what follows, we study the trace of $\mathcal{F}_\chi$. We first discuss the kernel and the image of $\mathcal{F}_\chi$. For a subset $K \subset A$, define $P_K: L^2(A) \to L^2(A)$ by
\begin{equation*}
  P_K(f)(a) = \frac{1}{|K|} \sum_{x \in K} f(a x) \quad (f \in L^2(A), a \in A).
\end{equation*}
Suppose that $K$ is a subgroup of $A$. Then we can regard $L^2(A/K)$ as a subspace of $L^2(A)$ via the map induced from the quotient map $A \to A/K$. $P_K$ is a projection from $L^2(A)$ onto $L^2(A/K)$.

Now let $J := \Rad(\chi)$ and let $\overline{\chi}$ be the bicharacter of $A/J$ induced from $\chi$. Then the following lemma can be proved by direct computation.

\begin{lemma}
  \label{lem:fourier-factor}
  $\mathcal{F}_\chi^{} = \sqrt{|J|} \mathcal{F}_{\overline{\chi}}^{}\,P_J$.
\end{lemma}

Since $\overline{\chi}$ is non-degenerate, $\mathcal{F}_{\overline{\chi}}^{}: L^2(A/J) \to L^2(A/J)$ is bijective. Therefore, the image of $\mathcal{F}_\chi$ is $L^2(A/J)$ and the kernel of $\mathcal{F}_\chi$ is that of $P_J$. In particular, $\mathcal{F}_\chi$ is bijective if and only if $\chi$ is non-degenerate.

In view of the previous lemma, we set $F_\chi := |J|^{-\frac{1}{2}} \mathcal{F}_\chi$. Then, again by direct computation, we can verify the following lemma:

\begin{lemma}
  \label{lem:fourier-square}
  $F_\chi^2(f)(a) = P_J^{}(f)(a^{-1})$.
\end{lemma}

Let $V_\lambda$ be the eigenspace of $F_\chi$ with eigenvalue $\lambda$. Since $F_\chi^4$ acts on $L^2(A/J)$ as identity by the above lemma, we have a decomposition
\begin{equation*}
  L^2(A) = \Ker(F_\chi) \oplus V_{+1} \oplus V_{-1} \oplus V_{+{\rm i}} \oplus V_{-{\rm i}}
\end{equation*}
where ${\rm i} = \sqrt{-1}$. This implies that $\Trace(F_\chi) = r + s \sqrt{-1}$, where
\begin{equation*}
  r = \dim(V_{+1})       - \dim(V_{-1})       \text{\quad and \quad}
  s = \dim(V_{+{\rm i}}) - \dim(V_{-{\rm i}}).
\end{equation*}

It seems to be difficult to determine $r$ and $s_\chi$ explicitly, however, we can determine the parities of them. Let $d_{\pm} = \frac{1}{2} (|A/J| \pm \# \{ a \in A/J \mid a^2 = 1 \}) \in \mathbb{Z}$.

\begin{lemma}
  \label{lem:fourier-eigen}
  $r \equiv d_+$ and $s \equiv d_- \pmod{2}$.
\end{lemma}
\begin{proof}
  Let $W_\lambda$ be the eigenspace of $F_\chi^2$ with eigenvalue $\lambda$. Then we have decompositions $W_{+1} = V_{+1} \oplus V_{-1}$ and $W_{-1} = V_{+{\rm i}} \oplus V_{-{\rm i}}$. Therefore,
  \begin{equation*}
    r = \dim(V_{+1}) - \dim(V_{-1}) \equiv \dim(V_{+1}) + \dim(V_{-1}) = \dim(W_{+1}) \pmod{2}.
  \end{equation*}
  Similarly, $s \equiv \dim(W_{-1}) \pmod{2}$.

  In what follows, we show that $\dim(W_{\pm 1}) = d_{\pm}$. By Lemma~\ref{lem:fourier-square}, we have that the restriction of $F_\chi^2$ on $L^2(A/J)$ is represented by the permutation matrix corresponding to the permutation $\sigma$ on $A/J$ given by $\sigma(a) = a^{-1}$ ($a \in A$). Therefore, by easy combinatorial arguments,
  \begin{equation*}
    \dim(W_{+1}) = \text{(the number of cycles in the cycle decomposition of $\sigma$)} = d_+.
  \end{equation*}
  Now we immediately have $\dim(W_{-1}) = |A/J| - \dim(W_{+1}) = d_-$.
\end{proof}

Summarizing results in this subsection, we have the following theorem:

\begin{theorem}
  \label{thm:fourier-eigen}
  Let $A$ be a finite abelian group and let $\chi$ be a symmetric bicharacter of $A$ with radical $J$. Then $\Trace(\mathcal{F}_\chi) = \sqrt{|J|} \cdot (n_+ + n_-\sqrt{-1})$ for some $n_{\pm} \in \mathbb{Z}$. $n_+$ and $n_-$ satisfy the following congruence equation:
  \begin{equation*}
    n_{\pm} \equiv d_{\pm} := \frac{1}{2} (|A/J| \pm \# \{ a \in A/J \mid a^2 = 1 \}) \pmod{2}.
  \end{equation*}
\end{theorem}

If $A$ is odd, $d_{\pm} = \frac{1}{2} (|A/J| + 1)$. Therefore, the following corollary follows:

\begin{corollary}
  \label{cor:fourier-eigen-odd}
  Notations are as in Theorem~\ref{thm:fourier-eigen}. Suppose that $|A|$ is odd. Then:
  \begin{enumalph}
  \item $n_+$ is odd if and only if $|A/J| \equiv 1 \pmod{4}$.
  \item $n_-$ is odd if and only if $|A/J| \equiv 3 \pmod{4}$.
  \end{enumalph}
\end{corollary}

\section{Indicators of Tambara-Yamagami categories}
\label{sec:indicators}

\subsection{Definition}

Let $A$ be a finite abelian group. In \cite{MR1659954}, Tambara and Yamagami classified fusion categories with representatives of isomorphism classes of simple objects $A \sqcup \{ m \}$ satisfying fusion rules
\begin{equation}
  \label{eq:fusion-rules-TY}
  a \otimes b \cong a b, \quad
  a \otimes m \cong m \cong m \otimes a, \quad
  m \otimes m \cong \bigoplus_{x \in A} x \quad (a, b \in A).
\end{equation}
They showed that such categories are parametrized by pairs $(\chi, \tau)$ of a non-de\-gen\-er\-ate symmetric bicharacter $\chi$ of $A$ and a square root $\tau$ of $|A|^{-1}$. The corresponding category is denoted by $\TY(A, \chi, \tau)$ and defined as follows:

\begin{definition}[{\cite[Definition~3.1]{MR1659954}}]
  \label{def:TY}
  $\TY(A, \chi, \tau)$ is a skeletal category with objects finite direct sums of elements of $S := A \sqcup \{ m \}$. $\Hom$-sets between elements of $S$ are given by
  \begin{equation*}
    \Hom(s, s') = \left\{
      \begin{array}{cl}
        \mathbb{C} & \text{if $s = s'$}, \\ 0 & \text{otherwise},
      \end{array}
    \right.
  \end{equation*}
  and the compositions of morphisms are obvious ones. Tensor products of elements of $S$ are given by \eqref{eq:fusion-rules-TY} (but with $\cong$ replaced by $=$). The unit object is $1 \in A$. The left and the right unit constraints are identity morphisms. The associativity constraint $\Phi$ is determined by
  \begin{gather*}
    \Phi_{a, m, b} = \chi(a, b) \id_m: m \to m, \\
    \Phi_{m, a, m} = (\chi(a, x) \delta_{x, y} \id_x)_{x, y}: \bigoplus_{x \in A} x \to \bigoplus_{y \in A} y, \\
    \Phi_{m, m, m} = (\tau \chi(x, y)^{-1} \id_m)_{x, y}: \bigoplus_{x \in A} m \to \bigoplus_{y \in A} m,
  \end{gather*}
  where $a, b \in A$, and the other $\Phi_{s, t, u}$ ($s, t, u \in S$) are identity morphisms.
\end{definition}

Two Tambara-Yamagami categories $\TY(A, \chi, \tau)$ and $\TY(A', \chi', \tau')$ are monoid\-al\-ly equivalent if and only if $(A, \chi)$ and $(A', \chi')$ are isometry and $\tau = \tau'$.

Throughout this section, we fix a triple $(A, \chi, \tau)$. $\TY(A, \chi, \tau)$ is left rigid. The duality is described as follows: The dual object of $a \in A$ is $a^* := a^{-1}$ with morphisms $\id_1: 1 \to a \otimes a^*$ and $\id_1: a^* \otimes a \to 1$. The dual object of $m$ is $m^* := m$ with morphisms $\iota:  1 \to m \otimes m^*$ and $\tau^{-1} p: m^* \otimes m \to 1$ where $\iota: 1 \to m \otimes m$ is the injection and $p: m \otimes m \to 1$ is the projection. $\TY(A, \chi, \tau)$ is also right rigid, and hence it is a fusion category.

Note that the bidual $(-)^{**}$ is equal to the identity functor of $\TY(A, \chi, \tau)$. There exists a pivotal structure $j$ on $\TY(A, \chi, \tau)$ determined by $j_a = \id_a$ ($a \in A$) and $j_m = \sgn(\tau) \id_m$, where $\sgn$ means the sign of a real number. This gives the canonical pivotal structure on $\TY(A, \chi, \tau)$ in the sense that the pivotal dimension with respect to $j$ coincides with the Frobenius-Perron dimension:
\begin{equation*}
  \pdim_j(a) = \FPdim(a) = 1 \ (a \in A), \quad
  \pdim_j(m) = \FPdim(m) = \sqrt{|A|}.
\end{equation*}
Thus the global dimension of $\TY(A, \chi, \tau)$ is $2|A|$.

\subsection{Formula for simple objects}

We compute Frobenius-Schur indicators of simple objects of $\TY(A, \chi, \tau)$. The computation of indicators of $a \in A$, which is considered as an object of $\TY(A, \chi, \tau)$, is easy and can be done as follows:

\begin{theorem}
  \label{thm:fs-ind-0}
  $\nu_n(a) = \delta_{a^n, 1} \ (a \in A)$.
\end{theorem}
\begin{proof}
  It follows from the definition that $\Hom(1, a^{\otimes n})$ is $\mathbb{C}$ if $a^n = 1$ and zero otherwise. Hence $\nu_n(a) = 0$ unless $a^n = 1$. If $a^n = 1$, then $E_a^{(n)}$ given by~(\ref{eq:E-map}) with $V = a$ is identity, and hence $\nu_n(a) = 1$. Summarizing, we have the result.
\end{proof}

The problem is the computation of $\nu_n(m)$. Since $\TY(A, \chi, \tau)$ is not concrete, the space $\Hom(1, m^{\otimes n})$ is hard to analyze. We use formula~(\ref{eq:FS-ind-twist}) to avoid this difficulty. The Drinfeld center of $\TY(A, \chi, \tau)$ is studied by some authors including Izumi \cite{MR1832764} and Gelaki-Naidu-Nikshych \cite{GNN09}. We follow the formulation of \cite{GNN09} and obtain the following list of simple objects of the left Drinfeld center of $\TY(A, \chi, \tau)$. (Remark that they studied the right Drinfeld center in \cite{GNN09}.)
\begin{itemize}
\item $X_{a, \varepsilon} = (a, s_{a, \varepsilon})$, parametrized by pairs $(a, \varepsilon)$ of $a \in A$ and a square root $\varepsilon$ of $\chi(a, a)$. The half-braiding $s_{a, \varepsilon}: a \otimes (-) \to (-) \otimes a$ is determined by
  \begin{equation*}
    s_{a, \varepsilon}(m) = \varepsilon, \quad
    s_{a, \varepsilon}(x) = \chi(a, x) \quad (x \in A).
  \end{equation*}
\item $Y_{a, b} = (a \oplus b, t_{a, b})$, parametrized by unordered pairs $(a, b)$ of distinct elements of $A$. Under the identification
  \begin{align*}
    \Hom(&(a \oplus b) \otimes X, Y \otimes (a \oplus b)) \\ \cong &
    \begin{pmatrix}
      \Hom(a \otimes X, Y \otimes a) & \Hom(a \otimes X, Y \otimes b) \\
      \Hom(b \otimes X, Y \otimes a) & \Hom(b \otimes X, Y \otimes b)
    \end{pmatrix},
  \end{align*}
  the half-braiding $t_{a, b}: (a \oplus b) \otimes (-) \to (-) \otimes (a \oplus b)$ is determined by
  \begin{equation*}
    t_{a, b}(m) = \begin{pmatrix} 0 & 1 \\ \chi(a, b) & 0 \end{pmatrix}, \quad
    t_{a, b}(x) = \begin{pmatrix} \chi(b, x) & 0 \\ 0 & \chi(a, x) \end{pmatrix} \quad (x \in A).
  \end{equation*}
\item $Z_{\rho, \Delta} = (m, u_{\rho, \Delta})$, parametrized by pairs $(\rho, \Delta)$ of $\rho \in C(\chi)$, where $C(\chi)$ is the same meaning as in Section~\ref{sec:preliminaries}, and a square root $\Delta$ of $\tau \sum_{a \in A} \rho(a)$. The half-braiding $u_{\rho, \Delta}: m \otimes (-) \to (-) \otimes m$ is determined by
  \begin{equation*}
    u_{\rho, \Delta}(m) = \Delta \bigoplus_{x \in A} \rho(x^{-1}) \id_x: \bigoplus_{x \in A} x \to \bigoplus_{x \in A} x
  \end{equation*}
  and $u_{\rho, \Delta}(x) = \rho(x)^{-1}$ ($x \in A$).
\end{itemize}

Let $\theta$ be the canonical twist of the left Drinfeld center $\mathcal{Z}$ of $\TY(A, \chi, \tau)$. Suppose that $X = (V, e_V) \in \mathcal{Z}$ is a simple object. Since the quantum trace in $\mathcal{Z}$ coincides with the pivotal trace in $\mathcal{Z}$, $\theta_X \in \mathbb{C}$ is the unique element satisfying
\begin{equation*}
  (1 \mathop{\longrightarrow} V \otimes V^*
  \mathop{\longrightarrow}^{e_{V, V^*}} V^* \otimes V
  \mathop{\longrightarrow} 1) \cdot \theta_{X} = \FPdim(V).
\end{equation*}
By this observation, we have
\begin{equation*}
  \theta_{X_{a, \varepsilon}} = \chi(a, a), \quad
  \theta_{Y_{a, b}} = \chi(a, b) \text{\quad and \quad}
  \theta_{Z_{\rho, \Delta}} = \Delta.
\end{equation*}

Now we have sufficient data to compute $\nu_n(m)$. In what follows, let $\widehat{\rho}$ denote the Fourier transform of a function $\rho$ on $A$ associated with $\chi$.

\begin{theorem}
  \label{thm:fs-ind-1}
  If $n$ is odd, $\nu_n(m) = 0$. Fix $\rho \in C(\chi)$. If $n = 2k$ is even,
  \begin{equation}
    \label{eq:fs-ind-ty-1}
    \nu_n(m)
    = \frac{\sgn(\tau)^k}{\sqrt{|A|}} \sum_{a \in A} \widehat{\rho}(a)^k
    = \frac{\sgn(\tau)^k}{|A|^{\frac{1}{2}(k-1)}} \rho^{*k}(1).
  \end{equation}
\end{theorem}

As remarked by Izumi in \cite{MR1832764},
\begin{equation}
  \label{eq:fourier-cochain}
  \widehat{\rho}(a) = \widehat{\rho}(1) \rho(a)^{-1} \quad (a \in A).
\end{equation}
Therefore (\ref{eq:fs-ind-ty-1}) can be rewritten in the following form:
\begin{equation*}
  \nu_{2k}(m) = \frac{\sgn(\tau)^k \, \widehat{\rho}(1)^k}{\sqrt{|A|}} \sum_{a \in A} \rho(a)^{-k}.
\end{equation*}
The formula of this form will be used for explicit computations in Section~\ref{sec:computation}.

\begin{proof}
  Provided that $n$ is odd, $\Hom(1, m^{\otimes n}) = 0$, and hence $\nu_n(m) = 0$. We suppose that $n = 2k$ is even. By~(\ref{eq:FS-ind-twist}) and the above list of simple objects, we have
  \begin{equation*}
    \nu_n(m)
    = \frac{1}{2|A|} \sum_{\varphi \in C(\chi)} \left( \tau \sum_{x \in A} \varphi(x) \right)^k \sqrt{|A|}
    = \frac{\sgn(\tau)^k}{\sqrt{|A|}} \sum_{\varphi \in C(\chi)} \widehat{\varphi}(1)^k.
  \end{equation*}
  Define $\rho_a: A \to \mathbb{C}$ by $\rho_a(x) = \rho(x) \chi(x, a)^{-1}$. Then the assignment $a \mapsto \rho_a$ gives a bijection between $A$ and $C(\chi)$ since $\chi$ is non-degenerate. If $\varphi = \rho_a$,
  \begin{equation*}
    \widehat{\varphi}(1) = \frac{1}{\sqrt{|A|}} \sum_{x \in A} \rho(x) \chi(x, a)^{-1} = \widehat{\rho}(a).
  \end{equation*}
  Thus the first equality of~(\ref{eq:fs-ind-ty-1}) follows. By the orthogonality relation of characters,
  \begin{equation*}
    \sum_{a \in A} \widehat{f}(a)
    = \frac{1}{\sqrt{|A|}} \sum_{a, x \in A} f(x) \chi(x, a)^{-1}
    = \sqrt{|A|} \, f(1) \quad (f \in L^2(A)).
  \end{equation*}
  By using this observation and~(\ref{eq:fourier-convolution}),
  \begin{equation*}
    \frac{\sgn(\tau)^k}{\sqrt{|A|}} \sum_{a \in A} \widehat{\rho}(a)^k
    = \frac{\sgn(\tau)^k}{|A|^{\frac{1}{2}k}} \sum_{a \in A} \mathcal{F}_\chi(\rho^{*k})(a)
    = \frac{\sgn(\tau)^k}{|A|^{\frac{1}{2}(k - 1)}} \rho^{*k}(1).
  \end{equation*}
  Thus the second equality of~(\ref{eq:fs-ind-ty-1}) follows.
\end{proof}

These formulae are useful for computations, however, it includes $\rho$ which does not appear in the definition of Tambara-Yamagami categories. There is a closed formula consisting only of $A$, $\chi$ and $\tau$, as follows:

\begin{theorem}
  \label{thm:fs-ind-ty-2}
  Suppose that $n = 2k$ is even. Then
  \begin{equation*}
    \nu_{n}(m) = \frac{\sgn(\tau)^k}{|A|^{\frac{1}{2}(k - 1)}}
    \sum_{a_1 \cdots a_k = 1} \, \prod_{1 \le i < j \le k} \chi(a_i, a_j)
  \end{equation*}
  where the sum runs through all $(a_1, \cdots, a_k) \in A^k$ such that $a_1 \cdots a_k = 1$.
\end{theorem}
\begin{proof}
  This follows immediately from the formula
  \begin{equation*}
    (f_1 * \cdots * f_k)(a) = \sum_{a_1 \cdots a_k = a} f_1(a_1) \cdots f_k(a_k) \quad (f_i \in L^2(A), a \in A)
  \end{equation*}
  and the second equality of Theorem~\ref{thm:fs-ind-ty-2}.
\end{proof}

In particular, $\nu_2(m)$ is the sign of $\tau$. By~(\ref{eq:fourier-trace}),
\begin{equation*}
  \nu_4(m) = \frac{1}{\sqrt{|A|}} \sum_{a \in A} \chi(a, a)^{-1} = \Trace(\mathcal{F}_\chi),
\end{equation*}
the trace of the Fourier transform associated with $\chi$. It is natural to ask what more higher indicators mean. The author expect that higher Frobenius-Schur indicators of $m$ can be interpreted in terms of Fourier analysis on finite groups in a natural way.

\subsection{Arithmetic properties}

Our formulae of indicators of $m \in \TY(A, \chi, \tau)$ are of complicated form. In this subsection, we study arithmetic properties of indicators of $m$ and show the following more familiar theorem.

For a finite group $G$ and an integer $n$, let $G[n] = \{ g \in G \mid g^n = 1 \}$.

\begin{theorem}
  \label{thm:fs-ind-arith}
  $\nu_{2k}(m) = \sqrt{|A[k]|} \cdot \xi$ for some $\xi \in \mu_8 \cup \{ 0 \}$.
  \begin{enumalph}
  \item Fix $\rho \in C(\chi)$. $\xi = 0$ if and only if there exists $a \in A[k]$ such that $\rho(a)^k \ne 1$.
  \item Suppose that $|A|$ is odd. Then
    \begin{equation*}
      \xi^2 =
      \begin{cases}
        +1 & \text{if $|A|^{k - 1}/|A[k]| \equiv 1 \pmod{4}$}, \\
        -1 & \text{otherwise}.
      \end{cases}
    \end{equation*}
  \end{enumalph}
\end{theorem}

Fix $k \ge 1$. We note that the following two conditions are equivalent:
\begin{enumerate}
\item There exist $\rho \in C(\chi)$ and $a \in A[k]$ such that $\rho(a)^k \ne 1$.
\item There exists $a \in A[k]$ such that $\rho(a)^k \ne 1$ for all $\rho \in C(\chi)$.
\end{enumerate}
In fact, since the character group $A^\vee$ acts transitively on $C(\chi)$ by the multiplication of functions on $A$, $\rho(a)^k = \rho'(a)^k$ for all $\rho, \rho' \in C(\chi)$ and $a \in A[k]$.

For a while, we fix a finite abelian group $B$ and a symmetric bicharacter $\beta$ of $B$. Let $\rho \in C(\beta)$. If $r \in \Rad(\beta)$ and $b \in B$, then $\rho(r b) = \rho(r) \rho(b)$. We remark that, in particular, the restriction of $\rho$ to $\Rad(\beta)$ is a character of $\Rad(\beta)$.

\begin{lemma}
  \label{lem:fs-ind-arith}
  Notations are as above. Let $J$ denote the radical of $\beta$.
  \begin{enumalph}
  \item Suppose that $\rho|_J$ is trivial. Then $\mathcal{F}_\beta(\rho)$ takes values in $\sqrt{|J|} \cdot \mu_\infty$.
  \item Otherwise $\mathcal{F}_\beta(\rho)$ is identically zero.
  \end{enumalph}
\end{lemma}
\begin{proof}
  We write $\mathcal{F}_\beta(\rho)$ by $\widehat{\rho}$ during the proof. Our proof is divided into two steps. We first prove the claim under the assumption that $\beta$ is non-degenerate.

  {\bf Step 1}. Suppose that $\beta$ is non-degenerate. Then $J$ is trivial. Recall that $\rho$ takes values in $\mu_\infty$. In view of~(\ref{eq:fourier-cochain}), it suffices to show that $\widehat{\rho}(1) \in \mu_\infty$. This can be proved by using Vafa's theorem (\cite{MR944264}, see also \cite[Theorem~3.1.19]{MR1797619}) which states that the order of the twist of a modular tensor category is finite.

  Consider the Tambara-Yamagami category $\mathcal{C} = \TY(B, \beta, |B|^{-1/2})$. Fix a square root $\Delta$ of $\widehat{\rho}(1) = |B|^{-1/2} \sum_{b \in B} \rho(b)$. Following the list of isomorphism classes of simple objects of $\mathcal{Z}(\mathcal{C})$, we obtain a simple object $Z \in \mathcal{Z}(\mathcal{C})$ such that $\theta_Z = \Delta$, where $\theta$ is the canonical twist of $\mathcal{Z}(\mathcal{C})$. It follows from Vafa's theorem that $\theta_Z$ is a root of unity, and hence so is $\widehat{\rho}(1) = \theta_Z^2$.

  {\bf Step 2}. We now consider the general case. Let $\overline{\beta}$ be the non-degenerate symmetric bicharacter of $B/J$ induced from $\beta$. By Lemma~\ref{lem:fourier-factor},
  \begin{equation*}
    \widehat{\rho} = \sqrt{|J|} \mathcal{F}_{\overline{\beta}}^{}(\overline{\rho})
    \text{\quad where \quad}
    \overline{\rho}(b) = \frac{1}{|J|} \sum_{r \in J} \rho(r b) \quad (b \in B).
  \end{equation*}

  {\rm (a)} Suppose that $\rho|_J$ is trivial. If we regard $\overline{\rho}$ as a function on $B/J$, then it is clear that $\overline{\rho} \in C(\overline{\beta})$. By Step~1, we have that $\mathcal{F}_{\overline{\beta}}^{}(\overline{\rho})$ takes values in $\mu_\infty$. Hence the result follows.

  {\rm (b)} Otherwise, if $\rho|_J$ is not trivial, it follows from the orthogonality relation of characters that $\overline{\rho}$ is identically zero. Thus also $\widehat{\rho}$ is identically zero.
\end{proof}

To apply the above lemma, we introduce some notations. Fix $k \ge 1$. Let
\begin{equation*}
  \mathcal{A}_k = \{ (a_1, \cdots, a_k) \in A^k \mid a_1 \cdots a_k = 1 \}.
\end{equation*}
$\mathcal{A}_k$ is a subgroup of $A^k$. Fix $\rho \in C(\chi)$ and define $\rho_k: \mathcal{A}_k \to \mathbb{C}^\times$ by
\begin{equation*}
  \rho_k(a_1, \cdots, a_k) = \rho(a_1) \cdots \rho(a_k) \quad (a_1, \cdots, a_k \in A)
\end{equation*}
Then $\chi_k := \partial \rho_k$ is a symmetric bicharacter of $\mathcal{A}_k$. $\chi_k$ is not necessarily non-de\-gen\-er\-ate. Let $J_k$ denote the radical of $\chi_k$. By direct computation, we have
\begin{equation*}
  J_k = \{ (a, a, \cdots, a) \in A^k \mid a^k = 1 \}.
\end{equation*}

Now we can prove the following weaker version of Theorem~\ref{thm:fs-ind-arith}.

\begin{proposition}
  \label{prop:fs-ind-arith}
  $\nu_{2k}(m) = \sqrt{|A[k]|} \cdot \xi$ for some $\xi \in \mu_{\infty} \cup \{ 0 \}$. $\xi = 0$ if and only if there exists $a \in A[k]$ such that $\rho(a)^k \ne 1$.
\end{proposition}

In particular, $\nu_{2k}(m) \ne 0$ if $A$ is odd.

\begin{proof}
  $\rho_k \in C(\chi_k)$ by definition. By Theorem~\ref{thm:fs-ind-1}, we have
  \begin{equation*}
    \nu_{2k}(m) = \frac{\sgn(\tau)^k}{\sqrt{|\mathcal{A}_k|}} \sum_{a \in \mathcal{A}_k} \rho_k(a) = \sgn(\tau)^k \cdot \mathcal{F}_{\chi_k}^{}(\rho_k)(1).
  \end{equation*}
  Now we can apply Lemma~\ref{lem:fs-ind-arith} to the right-hand side. Since $|J_k| = |A[k]|$, we have that $\nu_{2k}(m) = \sqrt{|A[k]|} \cdot \xi$ for some $\xi \in \mu_\infty \cup \{ 0 \}$. This $\xi$ is zero if and only if the restriction of $\rho_k$ to $J_k$ is not trivial. This condition is equivalent to that there exists an element $a \in A[k]$ such that $\rho(a)^k \ne 1$.
\end{proof}

To investigate the order of the above $\xi$, we introduce the following notation: For a finite abelian group $B$ and a symmetric bicharacter $\beta$ of $B$, we write
\begin{equation*}
  \Xi_k(B, \beta) = \frac{1}{|B|^{\frac{1}{2}(k - 1)} \cdot |B[k]|^{\frac{1}{2}}} \sum_{b_1 \cdots b_k = 1} \prod_{1 \le i < j \le k} \beta(b_i, b_j).
\end{equation*}
By Theorem~\ref{thm:fs-ind-ty-2},
\begin{equation*}
  \nu_{2k}(m) = \sgn(\tau)^k \sqrt{|A[k]|} \cdot \Xi_k(A, \chi).
\end{equation*}
Proposition~\ref{prop:fs-ind-arith} states that, if $\beta$ is non-degenerate, then $\Xi_k(B, \beta) \in \mu_{\infty} \cup \{ 0 \}$. Note that $\Xi_k(B, \beta)$ is an isometry invariant of $(B, \beta)$. Moreover, it is multiplicative in the following sense: If $(B, \beta)$ is isometric to $(B_1, \beta_1) \times (B_2, \beta_2)$, we have
\begin{equation*}
  \Xi_k(B, \beta) = \Xi_k(B_1, \beta_1) \cdot \Xi_k(B_2, \beta_2).
\end{equation*}

We first deal with the case where $|A|$ is odd. We can prove the following lemma by using results on the trace of the Fourier transform.

\begin{lemma}
  Suppose that $|A|$ is odd. Then
  \begin{equation*}
    \Xi_k(A, \chi)^2 =
    \begin{cases}
      +1 & \text{if $|A|^{k-1} / |A[k]| \equiv 1 \pmod{4}$}, \\
      -1 & \text{otherwise}.
    \end{cases}
  \end{equation*}
\end{lemma}
\begin{proof}
  $\mathcal{A}_k$, $J_k$ and $\rho_k$ have meanings as in the proof of Proposition~\ref{prop:fs-ind-arith}. Let $N$ be the exponent of $A$. Since $|A|$ is odd, so is $N$. Let $h = (N - 1)/2 \in \mathbb{Z}$. Define a symmetric bi\-char\-ac\-ter $\chi'_k: \mathcal{A}_k \times \mathcal{A}_k \to \mathbb{C}^\times$ by
  \begin{equation*}
    \chi'_k(a_1, \cdots, a_k; b_1, \cdots, b_k) = \prod_{1 \le i < j \le k} \chi(a_i, b_j)^h \chi(a_j, b_i)^h
    \quad (a_i, b_j \in A).
  \end{equation*}
  Then $\chi_k'(a, a)^{-1} = \rho_k(a)$ for all $a \in \mathcal{A}_k$. By Theorem~\ref{thm:fs-ind-1}, we have
  \begin{equation*}
    \sqrt{|A[k]|} \cdot \Xi_k(A, \chi)
    = \frac{1}{\sqrt{|\mathcal{A}_k|}} \sum_{a \in \mathcal{A}_k} \chi_k'(a, a)^{-1}
    = \Trace(\mathcal{F}_{\chi'_k}).
  \end{equation*}
  Note that $\Rad(\chi_k') = J_k \cong A[k]$. By Theorem~\ref{thm:fourier-eigen} states that
  \begin{equation*}
    \Trace(\mathcal{F}_{\chi'_k}) = \sqrt{|A[k]|} \cdot (r + s \sqrt{-1})
  \end{equation*}
  for some $r, s \in \mathbb{Z}$. Hence, we have $\Xi_k(A, \chi) = r + s \sqrt{-1}$. On the other hand, we know that $\Xi_k(A, \chi)$ is a root of unity. Therefore, $\Xi_k(A, \chi)$ is one of $\pm 1, \pm \sqrt{-1}$.

  Note that $|\mathcal{A}_k/J_k| = |A|^{k-1}/|A[k]|$. We can know parities of $r$ and $s$ by Corollary~\ref{cor:fourier-eigen-odd}. If $|A|^{k-1}/|A[k]| \equiv 1 \pmod{4}$, then $r$ is odd, and hence $\Xi_k(A, \chi)$ must be $\pm 1$. Otherwise, $s$ is odd, and hence $\Xi_k(A, \chi)$ must be $\pm \sqrt{-1}$.
\end{proof}

The definition of $\chi_k'$ in the above proof seems to be technical. Moreover, $\chi_k'$ cannot be defined if $|A|$ is even. In the case where $|A|$ is a 2-group, we can prove the following lemma by using Galois theory.

\begin{lemma}
  Suppose that $|A|$ is a 2-group. Then $\Xi_k(A, \chi) \in \mu_8 \cup \{ 0 \}$.
\end{lemma}
\begin{proof}
  Write $\xi = \Xi_k(A, \chi)$ for convention. We may assume that $\xi \ne 0$. Let $N$ be the exponent of $A$. Since $A$ is a 2-group, $N$ is a power of two. Fix a primitive $N$-th root $\zeta$ of unity. Then
  \begin{equation*}
    \Xi_k(A, \chi) \in \mathbb{Q}(\sqrt{2}, \zeta)
    \subset \mathbb{Q}\left( \frac{1+\sqrt{-1}}{\sqrt{2}}, \zeta \right)
    \subset \mathbb{Q}(\zeta'),
  \end{equation*}
  where $\zeta'$ is a primitive $8 N$-th root of unity. Thus the order of $\xi$ is a power of two, say $2^r$. Now we consider the Galois group ${\rm Gal}(\mathbb{Q}(\xi)/\mathbb{Q})$. As is well-known,
  \begin{equation*}
    {\rm Gal}(\mathbb{Q}(\xi)/\mathbb{Q}) \cong \mathbb{Z}_{2^r}^\times \cong
    \begin{cases}
      1                                        & \text{if $r \le 1$}, \\
      \mathbb{Z}_2                             & \text{if $r = 2$},   \\
      \mathbb{Z}_2 \times \mathbb{Z}_{2^{r-2}} & \text{otherwise}.
    \end{cases}
  \end{equation*}
  On the other hand, if $\sigma \in {\rm Gal}(\mathbb{Q}(\xi)/\mathbb{Q})$, then we have
  \begin{equation*}
    \sigma^2(\xi) = \sigma^2\left(\Xi_k(A, \chi)\right) = \Xi_k(A, {}^{\sigma\sigma}\chi) = \Xi_k(A, \chi) = \xi
  \end{equation*}
  by Lemma~\ref{lem:bichar-gal} and the isometry invariance. Therefore $r \le 3$, and hence $\xi^8 = 1$.
\end{proof}

\begin{proof}[Proof of Theorem~\ref{thm:fs-ind-arith}]
  This is a consequence of the previous two lemmas. Let $A_1$ be the set of all elements of $A$ of odd orders and let $A_2$ be the Sylow 2-subgroup of $A$. By the fundamental theorem of finite abelian groups and the Chinese remainder theorem, we have a decomposition $(A, \chi) \cong (A_1, \chi_1) \times (A_2, \chi_2)$ where $\chi_i$ is the restriction of $\chi$ to $A_i$. By the multiplicativity,
  \begin{equation*}
    \nu_{2k}(m) = \sgn(\tau)^k \sqrt{|A[k]|} \cdot \Xi_k(A_1, \chi_1) \cdot \Xi_k(A_2, \chi_2).
  \end{equation*}
  It follows from the previous two lemmas that $\xi \in \mu_8 \cup \{ 0 \}$. Also (a) and (b) follow from previous lemmas.
\end{proof}

\begin{remark}
  {\rm (i)} In general, it is difficult to determine $\Xi_k(A, \chi)$ explicitly. This problem involves the determination of quadratic Gauss sums; If $A = \mathbb{Z}_N$ and $\chi(i, j) = \zeta^{- i j}$ ($i, j \in A$), where $\zeta$ is a primitive $N$-th root of unity,
  \begin{equation*}
    \Xi_2(A, \chi) = \frac{1}{\sqrt{N}} \sum_{i = 0}^{N - 1} \zeta^{i^2} \times
    \begin{cases}
      1                  & \text{if $|A|$ is odd}, \\
      \frac{1}{\sqrt{2}} & \text{otherwise}.
    \end{cases}
  \end{equation*}

  {\rm (ii)} $\Xi_k(A, \chi)$ can be a primitive eighth root of unity. For example, if $A = \mathbb{Z}_8$ and $\chi(i, j) = \zeta^{-ij}$ ($i, j \in A$) with $\zeta = e^{\pi\sqrt{-1}/4}$, then $\Xi_2(A, \chi) = \zeta$.
\end{remark}

\section{On certain sums of indicators}
\label{sec:indicator-sum}

We denote by $\nu_n(H)$ the $n$-th Frobenius-Schur indicator of the regular representation of a semisimple quasi-Hopf algebra $H$. These numbers are interesting in the representation theory of Hopf algebras in view of their monoidal Morita invariance: Let $H$ and $L$ be two semisimple quasi-Hopf algebras. If $\Rep(H)$ and $\Rep(L)$ are monoidally equivalent, then $\nu_n(H) = \nu_n(L)$ for all $n$.

For a pivotal fusion category $\mathcal{C}$, we define $\nu_n(\mathcal{C})$ by
\begin{equation*}
  \nu_n(\mathcal{C}) = \sum_{V \in \Irr(\mathcal{C})} \nu_n(V) \pdim(V).
\end{equation*}
Then $\nu_n(\mathcal{C}) = \nu_n(H)$ if $\mathcal{C} = \Rep(H)$, see \cite{S10}. The results of Section~\ref{sec:indicators} yield formulae for $\nu_n(\mathcal{C})$ where $\mathcal{C}$ is a Tambara-Yamagami category, as follows.

\begin{theorem}
  \label{thm:fs-ind-sum-1}
  Let $\mathcal{C} = \TY(A, \chi, \tau)$ be a Tambara-Yamagami category.
  \begin{enumalph}
  \item If $n$ is odd, $\nu_n(\mathcal{C}) = |A[n]|$.
  \item If $n = 2k$ is even,
    \begin{equation*}
      \nu_{2k}(\mathcal{C}) = |A[2k]|
      + \frac{\sgn(\tau)^k}{|A|^{\frac{1}{2}k - 1}}
      \sum_{a_1 \cdots a_{k} = 1} \prod_{1 \le i < j \le k} \chi(a_i, a_j).
    \end{equation*}
  \end{enumalph}
\end{theorem}

We also have:

\begin{theorem}
  \label{thm:fs-ind-sum-2}
  Let $\mathcal{C} = \TY(A, \chi, \tau)$. Then
  \begin{equation*}
    \nu_{2k}(\mathcal{C}) = (2^r + \sqrt{|A/A[k]|} \cdot \xi) \times |A[k]|
  \end{equation*}
  for some $r \ge 0$ and $\xi \in \mu_8 \cup \{ 0 \}$. If $|A|$ is odd, $r = 0$ and
  \begin{equation*}
    \xi^2 =
    \begin{cases}
      +1 & \text{if $|A|^{k - 1}/|A[k]| \equiv 1 \pmod{4}$}, \\
      -1 & \text{otherwise}.
    \end{cases}
  \end{equation*}
\end{theorem}
\begin{proof}
  Since the exponent of $A[2k]/A[k]$ is less or equal to two, $|A[2k]/A[k]| = 2^r$ for some $r \ge 0$. By the definition of $\nu_{2k}(\mathcal{C})$ and Theorem~\ref{thm:fs-ind-arith},
  \begin{equation*}
    \nu_{2k}(\mathcal{C}) = |A[2k]| + \sqrt{|A[k]|} \cdot \xi \cdot \sqrt{|A|}
    = (2^r + \sqrt{|A/A[k]|} \cdot \xi) |A[k]|.
  \end{equation*}
  for some $\xi \in \mu_8 \cup \{ 0 \}$. If $|A|$ is odd, $r = 0$ since $A[2k] = A[k]$.
\end{proof}

If $G$ is a finite group, then $\nu_n(\mathbb{C}G) = |G[n]|$. Frobenius proved that $|G[n]|$ is divisible by $n$ if $n$ divides $|G|$. We call this fact {\em Frobenius theorem for finite groups}. One might ask whether $\nu_n(\mathcal{C})$ has the same property. Motivated by this question, the author introduced the following definition in \cite{S10}.

\begin{definition}
  \label{def:frobenius}
  Let $\mathcal{C}$ be a pivotal fusion category such that $\dim(\mathcal{C})$ is a positive integer. We say that {\em Frobenius theorem holds for $\mathcal{C}$} if $\nu_n(\mathcal{C})/n$ is an algebraic integer for every $n | \dim(\mathcal{C})$.
\end{definition}

We also say that Frobenius theorem holds for a semisimple quasi-Hopf algebra $H$ if it holds for $\mathcal{C} = \Rep(H)$. In \cite{S10}, the author showed that Frobenius theorem holds for various semisimple Hopf algebras and conjectured that it holds for every semisimple Hopf algebra. The author also pointed out that there exist many semisimple quasi-Hopf algebras for which Frobenius theorem does not hold.

It is interesting to investigate whether Frobenius theorem holds for Tambara-Yamagami categories. By results of Section~\ref{sec:indicators}, we can prove the following:

\begin{theorem}
  \label{thm:TY-frobenius}
  Let $\mathcal{C} = \TY(A, \chi, \tau)$.
  \begin{enumalph}
  \item Frobenius theorem holds for $\mathcal{C}$ if $A \equiv 1 \pmod{4}$.
  \item Frobenius theorem does not hold for $\mathcal{C}$ if $|A| \equiv 3 \pmod{4}$. In fact, $\nu_2(\mathcal{C})$ is not divisible by two.
  \end{enumalph}
\end{theorem}
\begin{proof}
  We first remark the following fact in the algebraic number theory. Let $d$ be an odd integer (which may not be square-free). Then $(1 \pm \sqrt{d})/2$ are algebraic integers if and only if $d \equiv 1 \pmod{4}$.

  {\rm (a)} Let $n$ be a divisor of $\dim(\mathcal{C}) = 2|A|$. Suppose that $n$ is odd. Then, since $|A|$ is odd, $n$ is a divisor of $|A|$. Hence Frobenius theorem for finite groups yields that $\nu_n(\mathcal{C}) = |A[n]|$ is divisible by $n$.

  Next, we suppose that $n = 2k$ is even. Then $k$ is a divisor of $|A|$. Put
  \begin{equation*}
    d = |A/A[k]| \times
    \begin{cases}
      +1 & \text{if $|A/A[k]| \equiv 1 \pmod{4}$}, \\
      -1 & \text{otherwise}.
    \end{cases}
  \end{equation*}
  Then $d \equiv 1 \pmod{4}$. Note that $|A|^{k-1}/|A[k]| \equiv |A/A[k]| \pmod{4}$ by the assumption. Theorem~\ref{thm:fs-ind-sum-2} states that $\nu_n(\mathcal{C})/n$ is of the form
  \begin{equation*}
    \frac{\nu_{n}(\mathcal{C})}{n}
    = \frac{|A[k]|}{k} \cdot \frac{1 \pm \sqrt{d}}{2}.
  \end{equation*}
  This is an algebraic integer by Frobenius theorem for finite groups and the above remark.

  {\rm (b)} By Theorem~\ref{thm:fs-ind-sum-2}, $\nu_2(\mathcal{C})/2$ is of the form $(1 \pm \sqrt{|A|})/2$. Hence this is not an algebraic integer by the above remark.
\end{proof}

\begin{corollary}
  \label{cor:TY-frobenius}
  Let $H$ be a quasi-Hopf algebra such that $\Rep(H)$ is monoidally equivalent to $\TY(A, \chi, \tau)$. Suppose that $\dim(H) = 2N$ with $N$ odd. Then Frobenius theorem holds for $H$.
\end{corollary}
\begin{proof}
  It suffices to show that $N = |A| \equiv 1 \pmod{4}$. The Frobenius-Perron dimension of $m \in \TY(A, \chi, \tau)$ is $\sqrt{|A|}$. On the other hand, it must be an integer under our assumption. Thus $|A| = (\sqrt{|A|})^2 \equiv 1 \pmod{4}$.
\end{proof}

\begin{remark}
  Theorem~\ref{thm:TY-frobenius} and Corollary~\ref{cor:TY-frobenius} cover cases where there does not exist a semisimple Hopf algebra $H$ such that $\mathcal{C} = \TY(A, \chi, \tau)$ is monoidally equivalent to $\Rep(H)$.

  Suppose that there exists such a semisimple Hopf algebra $H$. Then Frobenius theorem holds for $\mathcal{C}$. This can be proved without using any results of Section~\ref{sec:indicators}, as follows: Natale \cite{MR2016657} showed that, if this is the case, $H$ fits into a central extension
  \begin{equation*}
    1 \longrightarrow \mathbb{C}^{\mathbb{Z}_2} \longrightarrow H \longrightarrow \mathbb{C}A \longrightarrow 1.
  \end{equation*}
  The author has showed that Frobenius theorem holds for such an $H$ in \cite{S10}.
\end{remark}

\section{Hopf algebra with the traceless antipode}
\label{sec:traceless}

Kashina, Montgomery and Ng concerned in \cite{KMN09} whether there exists a semisimple Hopf algebra whose antipode has trace zero. In this section, we give an affirmative answer to this problem by using Tambara-Yamagami categories.

Let $H$ be a semisimple Hopf algebra with the antipode $S$. We note that the trace of $S$ is equal to $\nu_2(H)$, the second Frobenius-Schur indicator of the regular representation of $H$. Suppose that $\Rep(H)$ is monoidally equivalent to $\TY(A, \chi, \tau)$ for some $A$, $\chi$ and $\tau$. By using our formula, $\Trace(S)$ can be computed as follows:
\begin{equation}
  \label{eq:trace-antipode}
  \Trace(S) = \nu_2(H) = \nu_2(\TY(A, \chi, \tau)) = |A[2]| + \sgn(\tau) \sqrt{|A|}.
\end{equation}

Let $\mathcal{C}$ be a fusion category. It is known that there exists a semisimple Hopf algebra $H$ such that $\mathcal{C}$ is monoidally equivalent to $\Rep(H)$ if and only if $\mathcal{C}$ admits a {\em fiber functor}, that is, a $\mathbb{C}$-linear faithful exact monoidal functor from $\mathcal{C}$ to the category of finite-dimensional vector spaces over $\mathbb{C}$.

Fiber functors of $\mathcal{C} = \TY(A, \chi, \tau)$ was classified by Tambara in \cite{MR1776075}. Suppose that $|A|$ is square so that the Frobenius-Perron dimension of every object of $\mathcal{C}$ is an integer. Tambara showed that fiber functors of $\mathcal{C}$ is parametrized by pairs $(\sigma, \rho)$ of an involution $\sigma \in \Aut(A)$ and a map
\begin{equation*}
  \rho: V(\sigma) := \frac{\{ a \in A \mid \sigma(a) = a \}}{\{ a \cdot \sigma(a) \mid a \in A \}} \to \{ \pm 1 \}
\end{equation*}
satisfying the following conditions:
\begin{enumerate}
\item The bicharacter $A \times A \to \mathbb{C}^\times$, $(a, b) \mapsto \chi(a, \sigma(b))$ is alternating.
\item Let $\overline{\chi}$ be the bicharacter of $V(\sigma)$ induced from $\chi$. Then $\rho \in C(\overline{\chi})$ and the following equation holds:
  \begin{equation}
    \label{eq:sign-quadratic}
    \sgn(\tau) = \frac{1}{\sqrt{|V(\sigma)|}} \sum_{a \in V(\sigma)} \rho(a).
  \end{equation}
\end{enumerate}
The reader should refer to \cite{MR1776075} for further details. In the original paper, the right-hand side of equation~(\ref{eq:sign-quadratic}) appears as the sign of quadratic form on $V(\sigma)$, see \cite[Lemma~2.10]{MR1776075}.

Now we fix $r > 0$ and set $A = \mathbb{Z}_4^{2 r}$. Then $|A[2]| = \sqrt{|A|} = 2^{2r}$.

\begin{lemma}
  \label{lem:fiber-functor}
  Define $\chi: A \times A \to \mathbb{C}^\times$ by $\chi(a, b) = (\sqrt{-1})^{E(a, b)}$ where
  \begin{equation*}
    E(a_1, \cdots, a_{2r}; b_1, \cdots, b_{2r}) = \sum_{i = 1}^r (a_{2i - 1} b_{2i} + a_{2i} b_{2i - 1})
    \quad (a_j, b_j \in \mathbb{Z}_4).
  \end{equation*}
  Then $\TY(A, \chi, -2^{-2r})$ admits a fiber functor.
\end{lemma}

Hence there exists a semisimple Hopf algebra $H$ such that $\Rep(H)$ is monoidally equivalent to $\TY(A, \chi, -2^{-2r})$. It follows from~(\ref{eq:trace-antipode}) that the trace of the antipode of such an $H$ is zero.

\begin{proof}
  Define an involution $\sigma \in \Aut(A)$ by
  \begin{equation*}
    \sigma(a_1, \cdots, a_{2r}) = (-a_1, a_2, \cdots, (-1)^s a_s, \cdots, -a_{2r - 1}, a_{2r})
    \quad (a_i \in \mathbb{Z}_4).
  \end{equation*}
  Then $\chi(a, \sigma(a)) = 1$ for all $a \in A$. $V := V(\sigma)$ is computed as
  \begin{equation*}
    V = \frac{2 \mathbb{Z}_4}{0} \times \frac{\mathbb{Z}_4}{2 \mathbb{Z}_4}
    \times \cdots \times
    \frac{2 \mathbb{Z}_4}{0} \times \frac{\mathbb{Z}_4}{2 \mathbb{Z}_4} \quad (\text{$r$ times}).
  \end{equation*}
  Note that each component is isomorphic to the additive group of the field $\mathbb{F}_2$ of two elements. Let $f: 2 \mathbb{Z}_4 \to \mathbb{F}_2$ and $g: \mathbb{Z}_4/2\mathbb{Z}_4 \to \mathbb{F}_2$ be isomorphisms. In what follows, we identify $V$ with $\mathbb{F}_2^{2r}$ via the map $f \times g \times \cdots \times f \times g$.

  Let $\overline{\chi}$ be the bicharacter of $V$ induced from $\chi$. Under the above identification of $V$ with $\mathbb{F}_2^{2r}$, $\overline{\chi}$ is given by $\overline{\chi}(v, w) = (-1)^{B(v, w)}$ where
  \begin{equation*}
    B(x^{}_1, \cdots, x^{}_{2r}; y^{}_1, \cdots, y^{}_{2r})
    = \sum_{i = 1}^r (x^{}_{2i - 1} y^{}_{2i} + x^{}_{2i} y^{}_{2i - 1})
    \quad (x^{}_j, y^{}_j \in \mathbb{F}_2).
  \end{equation*}
  Define $q: V \to \mathbb{F}_2$ by
  \begin{equation*}
    q(x^{}_1, \cdots, x^{}_{2r}) = x^{}_1 + x^{}_2 + \sum_{i = 1}^{r} x^{}_{2i-1} x^{}_{2i}
    \quad (x_i \in \mathbb{F}_2).
  \end{equation*}
  Then we have that $B(x, y) = q(x) - q(x + y) + q(y)$ for all $x, y \in \mathbb{F}_2^{2r}$, and hence the function $\rho: V \to \{ \pm 1 \}$ given by $\rho(v) = (-1)^{q(v)}$ is in $C(\overline{\chi})$. Moreover,
  \begin{equation*}
    \frac{1}{\sqrt{|V|}} \sum_{v \in V} \rho(v)
    = \frac{1}{2^r} \left( \sum_{i, j = 0}^{1} (-1)^{i + j + ij} \right)
    \left( \sum_{i, j = 0}^{1} (-1)^{i j} \right)^{r - 1} = -1.
  \end{equation*}
  Now our claim follows from the classification result of fiber functors of Tambara-Yamagami categories.
\end{proof}

\section{Computational examples}
\label{sec:computation}

In this section, we determine Frobenius-Schur indicators of simple objects of Tambara-Yamagami categories associated with finite-dimensional vector spaces over $\mathbb{F}_p = \mathbb{Z}/p\mathbb{Z}$, where $p$ is a prime number.

\subsection{Odd characteristic}

We first deal with the case where $p$ is odd. Let $V = \mathbb{F}_p^r$ be an $r$-dimensional vector space over $\mathbb{F}_p$ with $e_1, \cdots, e_r$ the standard basis. If $B: V \times V \to \mathbb{F}_p$ is a bilinear form on $V$,
\begin{equation*}
  \chi_B^{}: V \to V \to \mathbb{C}^\times; \quad
  (v, w) \mapsto \exp\left(\frac{2\pi\sqrt{-1}}{p}B(v, w)\right) \quad (v, w \in V)
\end{equation*}
is a bicharacter on $V$. The assignment $B \mapsto \chi_B^{}$ gives a bijection between bilinear forms on $V$ and bicharacters of $V$. It is easy to see that $\chi_B^{}$ is a non-degenerate symmetric bicharacter of $V$ if and only if $B$ is a non-degenerate symmetric bilinear form.

We can identify a bilinear form $B$ on $V$ with the matrix $B = (B(e_i, e_j))$. We say that $\chi_B^{}$ is diagonal if $B$ is diagonal. Given a non-degenerate symmetric bicharacter $\chi = \chi_B^{}$ of $V$, define $D(\chi) \in \{ \pm 1 \}$ by
\begin{equation*}
  D(\chi) = \legendre{\det B}{p}
  := \begin{cases}
    +1 & \text{if $x^2 = \det B$ has a solution in $\mathbb{F}_p$}, \\
    -1 & \text{otherwise},
  \end{cases}
\end{equation*}
by using the Legendre symbol. It can be showed that $D(\chi)$ is a complete invariant of non-degenerate symmetric bicharacters of $V$ (see \cite[Chapter IV]{MR0344216}). In particular, every such bicharacter is isometric to a diagonal one.

We need the following well-known formula: If $a$ is relatively prime to $p$,
\begin{equation}
  \label{eq:gauss-sum}
  \sum_{i = 0}^{p - 1} \exp\left(\frac{2\pi\sqrt{-1}}{p} a i^2 \right)
  = \legendre{a}{p} \varepsilon_p \sqrt{p}
\end{equation}
where $\varepsilon_p$ indicates one if $p \equiv 1 \pmod{4}$ and $\sqrt{-1}$ if $p \equiv 3 \pmod{4}$.

\begin{theorem}
  \label{thm:FS-ind-comp-1}
  Let $\chi$ be a non-degenerate symmetric bicharacter of $V$ and $\tau$ a square root of $|V|$. The $2k$-th Frobenius-Schur indicator of $m \in \TY(V, \chi, \tau)$ is given as follows:
  \begin{enumalph}
  \item If $k$ is relatively prime to $p$,
    \begin{equation*}
      \nu_{2k}(m) = \sgn(\tau)^k \varepsilon_p^{r(k+1)} \legendre{-k}{p}^r \legendre{-2}{p}^{r(k+1)} D(\chi)^{k+1}.
    \end{equation*}
  \item If $k$ is a multiple of $p$,
    \begin{equation*}
      \nu_{2k}(m) = \sgn(\tau)^k \varepsilon_p^{r k} \legendre{-2}{p}^{r k} D(\chi)^k \times p^{\frac{1}{2}r}.
    \end{equation*}
  \end{enumalph}
\end{theorem}
\begin{proof}
  We may assume $\chi$ to be diagonal. There exists $a_s \in \mathbb{F}_p^\times$ such that
  \begin{equation*}
    \chi(i_1, \cdots, i_r; j_1, \cdots, j_r) = \exp\left(\frac{2\pi\sqrt{-1}}{p} \sum_{s = 1}^r a_s i_s j_s \right)
    \quad (i_s, j_s \in \mathbb{F}_p).
  \end{equation*}
  By the multiplicativity of the Legendre symbol, we have
  \begin{equation*}
    D(\chi) = \legendre{a_1 \cdots a_r}{p} = \legendre{a_1}{p} \cdots \legendre{a_r}{p}.
  \end{equation*}
  Let $h = (p-1)/2$ so that $2h = -1$ in $\mathbb{F}_p$. Define $\rho: V \to \mathbb{C}^\times$ by
  \begin{equation*}
    \rho(i_1, \cdots, i_r) = \exp\left(\frac{2\pi h \sqrt{-1}}{p} \sum_{s = 1}^r a_s^{} i_s^2 \right)
    \quad (i_s \in \mathbb{F}_p).
  \end{equation*}

  {\rm (a)} One can easily see that $\rho \in C(\chi)$. By~(\ref{eq:gauss-sum}),
  \begin{equation*}
    \frac{1}{\sqrt{|V|}} \sum_{v \in V} \rho(v)^t
    = \prod_{s = 1}^r \left\{ \frac{1}{\sqrt{p}} \sum_{i = 0}^{p - 1} \exp\left(\frac{2\pi\sqrt{-1}}{p} a_s h t i^2\right) \right\}
    = \legendre{-2t}{p}^r D(\chi) \, \varepsilon_p^r.
  \end{equation*}
  Letting $t = 1$, we have $\widehat{\rho}(0) = \legendre{-2}{p}^r D(\chi) \varepsilon_p^r$, where $\widehat{\rho}$ is the Fourier transform of $\rho$ associated with $\chi$. Now our claim follows immediately from Theorem~\ref{thm:fs-ind-1}.
 
  {\rm (b)} This can be proved in a similar way as (a).
\end{proof}

We have $\nu_2(m) = \sgn(\tau)$ and $\nu_4(m) = \varepsilon_p^{3r} D(\chi)$. Thus $\tau$ and the isometry class of $\chi$ can be recovered from Frobenius-Schur indicators. This observation gives the following corollary:

\begin{corollary}
  \label{cor:FS-ind-fin-vec-odd}
  Let $\mathcal{C} = \TY(V, \chi, \tau)$ and $\mathcal{D} = \TY(V, \chi', \tau')$ be Tambara-Yamagami categories associated with $V$. Then the following conditions are equivalent:
  \begin{enumalph}
  \item $\mathcal{C}$ and $\mathcal{D}$ are monoidally equivalent.
  \item $\nu_n(m \in \mathcal{C}) = \nu_n(m \in \mathcal{D})$ for $n = 2, 4$.
  \item $\nu_n(\mathcal{C}) = \nu_n(\mathcal{D})$ for $n = 2, 4$.
  \end{enumalph}
\end{corollary}

Masuoka showed that a semisimple Hopf algebra of dimension $2 p^2$ with $p^2$ non-equivalent one-dimensional representations is unique up to isomorphism \cite[Theorem~1.9]{MR1370538}. Let $H_{2p^2}$ be such a Hopf algebra. By computing fusion rules, we have that $\Rep(H)$ is monoidally equivalent to $\TY(\mathbb{F}_p^2, \chi, \tau)$ for some $\chi$ and $\tau$.

Fusion rules do not give any information about $\chi$ and $\tau$. We can determine the isometry class of $\chi$ and $\tau$ by using Corollary~\ref{cor:FS-ind-fin-vec-odd}. Frobenius-Schur indicators of the regular representation of $H_{2p^2}$ have been determined by the author \cite[Theorem~5.1]{S10} as follows:
\begin{equation*}
  \nu_n(H_{2p^2}) = \gcd(n, p)^2 +
  \begin{cases}
    0                  & \text{if $n$ is odd}, \\
    p \cdot \gcd(n, p) & \text{if $n$ is even}.
  \end{cases}
\end{equation*}
On the other hand,
\begin{equation*}
  \nu_2(\TY(\mathbb{F}_p^2, \chi, \tau)) = 1 + \sgn(\tau) \cdot p \text{\quad and \quad}
  \nu_4(\TY(\mathbb{F}_p^2, \chi, \tau)) = 1 + D(\chi) \varepsilon_p^2 \cdot p.
\end{equation*}
Comparing them with $\nu_n(H)$, we conclude that $\tau = +p^{-1}$ and that $\chi$ is a non-de\-gen\-er\-ate symmetric bicharacter of $\mathbb{F}_p^2$ such that $D(\chi) = \varepsilon_p^2$.

\subsection{Characteristic two}

Let $V = \mathbb{F}_2^r$ with $e_1, \cdots, e_r$ the standard basis. If $B$ is a bilinear form on $V$, then $\chi_B^{}(v, w) = (-1)^{B(v, w)}$ ($v, w \in V$) is a bicharacter of $V$. Same as the case where $p > 2$, the assignment $B \mapsto \chi_B^{}$ gives a bijection between bilinear forms on $V$ and bicharacters of $V$.

The classification of non-degenerate symmetric bilinear forms on $V$ is well-known (see, for example, \cite{MR1501952}). If $r$ is odd, every such bilinear form is isometric to the bilinear form determined by $B_{\rm sym}(e_i, e_j) = \delta_{i j}$. If $r$ is even, there exist exactly two isometry classes of such bilinear forms. One is represented by $B_{\rm sym}$ and the another is represented by
\begin{equation*}
  B_{\rm alt}(i_1, \cdots, i_r; j_1, \cdots, j_r) =
  \sum_{s = 1}^{r/2} (i_{2s - 1} j_{2s} + i_{2s} j_{2s - 1})
  \quad (i_s, j_s \in \mathbb{F}_2).
\end{equation*}
We denote by $\chi_*^{}$ ($* = {\rm sym}, {\rm alt}$) the bicharacter of $V$ corresponding to $B_*$.

\begin{theorem}
  \label{thm:FS-ind-comp-2}
  Let $\chi$ be a non-degenerate symmetric bicharacter of $V$ and $\tau$ a square root of $|V|$. The $2k$-th Frobenius-Schur indicator of $m \in \TY(V, \chi, \tau)$ is given as follows:
  \begin{enumalph}
  \item If $\chi$ is isometric to $\chi_{\rm sym}$,
    \begin{equation*}
      \nu_{2k}(m) = \sgn(\tau)^k
      \left(\frac{1 + \sqrt{-1}}{\sqrt{2}}\right)^{r k}
      \left(\frac{1 + (\sqrt{-1})^k}{\sqrt{2}}\right)^r
    \end{equation*}
  \item If $r$ is even and $\chi$ is isometric to $\chi_{\rm alt}$,
    \begin{equation*}
      \nu_{2k}(m) =
      \begin{cases}
        \sgn(\tau)       & \text{if $k$ is odd}, \\
        2^{\frac{1}{2}r} & \text{if $k$ is even}.
      \end{cases}
    \end{equation*}
  \end{enumalph}
\end{theorem}
\begin{proof}
  {\rm (a)} We may assume that $\chi = \chi_{\rm sym}$. Define $q: V \to \mathbb{Z}_{\ge 0}$ by
  \begin{equation*}
    q(i_1, \cdots, i_r) = \# \{ s = 1, \cdots, r \mid i_s = 1 \} \quad (i_s \in \mathbb{F}_2).
  \end{equation*}
  The function $\rho: V \to \mathbb{C}^\times$ given by $\rho(v) = (\sqrt{-1})^{q(v)}$ is in $C(\chi)$. For $t \in \mathbb{Z}$,
  \begin{equation*}
    \frac{1}{\sqrt{|V|}} \sum_{v \in V} \rho(v)^t
    = \frac{1}{\sqrt{|V|}} \sum_{s = 0}^{\infty} \# (q^{-1}(s)) \cdot (\sqrt{-1})^{s t} \\
    = \left(\frac{1 + (\sqrt{-1})^t}{\sqrt{2}}\right)^r
  \end{equation*}
  by the binomial theorem. The rest of the proof is same as that of Theorem~\ref{thm:FS-ind-comp-1}.

  {\rm (b)} We may assume that $\chi = \chi_{\rm alt}$. Define $q: V \to \mathbb{F}_2$ by
  \begin{equation*}
    q(i_1, \cdots, i_r) = \sum_{s = 1}^{r/2} i_{2s-1} i_{2s} \quad (i_s \in \mathbb{F}_2).
  \end{equation*}
  The function $\rho: V \to \mathbb{C}^\times$ given by $\rho(v) = (-1)^{q(v)}$ is in $C(\chi)$. In a similar way as the proof of Lemma~\ref{lem:fiber-functor}, we compute as follows:
  \begin{equation*}
    \frac{1}{\sqrt{|V|}} \sum_{v \in V} \rho(v)^t =
    \begin{cases}
      1                & \text{if $t$ is odd}, \\
      2^{\frac{1}{2}r} & \text{if $t$ is even}.
    \end{cases}
  \end{equation*}
  The rest of the proof is same as that of Theorem~\ref{thm:FS-ind-comp-1}.
\end{proof}

Similarly to Corollary~\ref{cor:FS-ind-fin-vec-odd}, we have the following:

\begin{corollary}
  \label{cor:FS-ind-fin-vec-even}
  Let $\mathcal{C} = \TY(V, \chi, \tau)$ and $\mathcal{D} = \TY(V, \chi', \tau')$ be Tambara-Yamagami categories associated with $V$. Then the following conditions are equivalent:
  \begin{enumalph}
  \item $\mathcal{C}$ and $\mathcal{D}$ are monoidally equivalent.
  \item $\nu_n(m \in \mathcal{C}) = \nu_n(m \in \mathcal{D})$ for $n = 2, 4$.
  \item $\nu_n(\mathcal{C}) = \nu_n(\mathcal{D})$ for $n = 2, 4$.
  \end{enumalph}
\end{corollary}

\begin{table}
  \begin{tabular}{cc|cccccccc}
    $\chi$           & $\tau$ & $\nu_1$ & $\nu_2$ & $\nu_3$ & $\nu_4$ & $\nu_5$ & $\nu_6$ & $\nu_7$ & $\nu_8$ \\
    \hline
    $\chi_{\rm alt}$ & $+1/2$ & $0$ & $+1$ & $0$ & $2$ & $0$ & $+1$ & $0$ & $2$ \\
    $\chi_{\rm alt}$ & $-1/2$ & $0$ & $-1$ & $0$ & $2$ & $0$ & $-1$ & $0$ & $2$ \\
    $\chi_{\rm sym}$ & $+1/2$ & $0$ & $+1$ & $0$ & $0$ & $0$ & $+1$ & $0$ & $2$ \\
    $\chi_{\rm sym}$ & $-1/2$ & $0$ & $-1$ & $0$ & $0$ & $0$ & $-1$ & $0$ & $2$ \\
  \end{tabular}
  \caption{Frobenius-Schur indicators $\nu_n$ ($n = 1, \cdots, 8$) of $m \in \TY(\mathbb{F}_2^2, \chi, \tau)$}
  \label{tab:fs-ind-m}
\end{table}

Letting $r = 2$ in Theorem~\ref{thm:FS-ind-comp-2}, we obtain Table~\ref{tab:fs-ind-m}. This table has been obtained by Ng and Schauenburg in \cite{MR2366965} by using central gauge invariants of quasi-Hopf algebras.

It is known that a non-trivial semisimple Hopf algebra of dimension eight is unique up to isomorphism. Let $\mathcal{B}_8$ be such a Hopf algebra. By computing fusion rules, we have that $\Rep(\mathcal{B}_8)$ is monoidally equivalent to $\TY(\mathbb{F}_2^2, \chi, \tau)$ for some $\chi$ and $\tau$. The isometry class of $\chi$ and $\tau$ can be determined by a similar method used in the previous subsection. As computed by the author \cite{S10},
\begin{equation*}
  \nu_n(\mathcal{B}_8) =
  \begin{cases}
    1 & \text{if $n \equiv 1, 3, 5, 7$}, \\
    6 & \text{if $n \equiv 2$}, \\
    4 & \text{if $n \equiv 4$}, \\
    8 & \text{if $n \equiv 8 \pmod{8}$}.
  \end{cases}
\end{equation*}
Thus, by the previous corollary, we have that $\Rep(\mathcal{B}_8) \approx \TY(\mathbb{F}_2^2, \chi_{\rm sym}, +1/2)$ as monoidal categories. This result has been showed by Tambara and Yamagami in \cite{MR1659954} by direct computation.

Of course, the same method can be applicable for group algebras of $D_8$, the dihedral group of order eight, and that of $Q_8$, the quaternion group of the same order. As showed in \cite{MR1659954},
\begin{equation*}
  \Rep(\mathbb{C}D_8) \approx \TY(\mathbb{F}_2^2, \chi_{\rm alt}, +1/2) \text{\quad and \quad}
  \Rep(\mathbb{C}Q_8) \approx \TY(\mathbb{F}_2^2, \chi_{\rm alt}, -1/2)
\end{equation*}
as monoidal categories.

\section*{Acknowledgements}

The author is supported by Grant-in-Aid for JSPS Fellows.

\bibliographystyle{abbrv}

\begin{thebibliography}{10}

\bibitem{MR1501952}
A.~A. Albert.
\newblock Symmetric and alternate matrices in an arbitrary field. {I}.
\newblock {\em Trans. Amer. Math. Soc.}, 43(3):386--436, 1938.

\bibitem{MR1797619}
B.~Bakalov and A.~Kirillov, Jr.
\newblock {\em Lectures on tensor categories and modular functors}, volume~21
  of {\em University Lecture Series}.
\newblock American Mathematical Society, Providence, RI, 2001.

\bibitem{MR1686423}
J.~W. Barrett and B.~W. Westbury.
\newblock Spherical categories.
\newblock {\em Adv. Math.}, 143(2):357--375, 1999.

\bibitem{MR2183279}
P.~Etingof, D.~Nikshych, and V.~Ostrik.
\newblock On fusion categories.
\newblock {\em Ann. of Math. (2)}, 162(2):581--642, 2005.

\bibitem{GNN09}
S.~Gelaki, D.~Naidu, and D.~Nikshych.
\newblock Centers of graded fusion categories.
\newblock {\em arXiv:0905.3117}, 2009.

\bibitem{MR1832764}
M.~Izumi.
\newblock The structure of sectors associated with {L}ongo-{R}ehren inclusions.
  {II}. {E}xamples.
\newblock {\em Rev. Math. Phys.}, 13(5):603--674, 2001.

\bibitem{KMN09}
Y.~Kashina, S.~Montgomery, and S.-H. Ng.
\newblock On the trace of the antipode and higher indicators.
\newblock {\em arXiv:0910.1628}.

\bibitem{MR1321145}
C.~Kassel.
\newblock {\em Quantum groups}, volume 155 of {\em Graduate Texts in
  Mathematics}.
\newblock Springer-Verlag, New York, 1995.

\bibitem{S10}
K.~Shimizu.
\newblock Some computations of Frobenius-Schur indicators of the regular
  representations of Hopf algebras.
\newblock {\em arXiv:1002.4086}, 2010.

\bibitem{MR0174656}
A.~Kleppner.
\newblock Multipliers on abelian groups.
\newblock {\em Math. Ann.}, 158:11--34, 1965.

\bibitem{MR1808131}
V.~Linchenko and S.~Montgomery.
\newblock A {F}robenius-{S}chur theorem for {H}opf algebras.
\newblock {\em Algebr. Represent. Theory}, 3(4):347--355, 2000.

\bibitem{MR2104908}
G.~Mason and S.-H. Ng.
\newblock Central invariants and {F}robenius-{S}chur indicators for semisimple
  quasi-{H}opf algebras.
\newblock {\em Adv. Math.}, 190(1):161--195, 2005.

\bibitem{MR1370538}
A.~Masuoka.
\newblock Some further classification results on semisimple {H}opf algebras.
\newblock {\em Comm. Algebra}, 24(1):307--329, 1996.

\bibitem{MR1966525}
M.~M{\"u}ger.
\newblock From subfactors to categories and topology. {II}. {T}he quantum
  double of tensor categories and subfactors.
\newblock {\em J. Pure Appl. Algebra}, 180(1-2):159--219, 2003.

\bibitem{MR2016657}
S.~Natale.
\newblock On group theoretical {H}opf algebras and exact factorizations of
  finite groups.
\newblock {\em J. Algebra}, 270(1):199--211, 2003.

\bibitem{MR2313527}
S.-H. Ng and P.~Schauenburg.
\newblock Frobenius-{S}chur indicators and exponents of spherical categories.
\newblock {\em Adv. Math.}, 211(1):34--71, 2007.

\bibitem{MR2381536}
S.-H. Ng and P.~Schauenburg.
\newblock Higher {F}robenius-{S}chur indicators for pivotal categories.
\newblock In {\em Hopf algebras and generalizations}, volume 441 of {\em
  Contemp. Math.}, pages 63--90. Amer. Math. Soc., Providence, RI, 2007.

\bibitem{MR2366965}
S.-H. Ng and P.~Schauenburg.
\newblock Central invariants and higher indicators for semisimple quasi-{H}opf
  algebras.
\newblock {\em Trans. Amer. Math. Soc.}, 360(4):1839--1860, 2008.

\bibitem{MR0344216}
J.-P. Serre.
\newblock {\em A course in arithmetic}.
\newblock Springer-Verlag, New York, 1973.
\newblock Translated from the French, Graduate Texts in Mathematics, No. 7.

\bibitem{MR1776075}
D.~Tambara.
\newblock Representations of tensor categories with fusion rules of
  self-duality for abelian groups.
\newblock {\em Israel J. Math.}, 118:29--60, 2000.

\bibitem{MR1659954}
D.~Tambara and S.~Yamagami.
\newblock Tensor categories with fusion rules of self-duality for finite
  abelian groups.
\newblock {\em J. Algebra}, 209(2):692--707, 1998.

\bibitem{MR944264}
C.~Vafa.
\newblock Toward classification of conformal theories.
\newblock {\em Phys. Lett. B}, 206(3):421--426, 1988.

\end{thebibliography}

\end{document}